\documentclass[11pt,reqno]{amsart}
\swapnumbers
\usepackage{a4,amssymb,amsthm,amscd,amsmath,verbatim,url,enumerate,mathdots, hyperref,
cleveref,mathtools}
\usepackage[pdftex,dvipsnames]{xcolor}
\usepackage{fancyhdr}
\usepackage{amsmath} 

\title{A CLASS OF SIMPLE DERIVATIONS OF POLYNOMIAL RING $k[x_1,x_2, \ldots ,x_n]$}

\author{Sumit Chandra Mishra}
\thanks{Corresponding author: \texttt{Sumit Chandra Mishra}}
\author{Dibyendu Mondal}
\author{Pankaj Shukla}

\address{Indian Institute of Technology Indore, Simrol, Khandwa Road, Indore 453 552 India}
\email{sumitcmishra@gmail.com, sumitcmishra@iiti.ac.in}

\address{Indian Institute of Technology Indore, Simrol, Khandwa Road, Indore 453 552 India}
\email{mdibyendu07@gmail.com, mdibyendu@iiti.ac.in}

\address{Indian Institute of Technology Indore, Simrol, Khandwa Road, Indore 453 552 India}
\email{Pankajshuklashvm23@gmail.com, phd2201141002@iiti.ac.in}

\date{\today}
\renewcommand{\deg}{\mathsf{deg}}

\renewcommand{\setminus}{\smallsetminus}

\renewcommand{\leq}{\leqslant}
\renewcommand{\geq}{\geqslant}

\swapnumbers

\newtheorem*{thm*}{Theorem}

\newtheorem{thm}{Theorem}
\numberwithin{thm}{section}
\newtheorem{prop}[thm]{Proposition}

\newtheorem{conj}[thm]{Conjecture}

\newtheorem{lem}[thm]{Lemma}

\theoremstyle{definition}

\newtheorem{nota}[thm]{Notation}

\newtheorem{rem}[thm]{Remark}

\numberwithin{equation}{thm}
\renewenvironment{proof}{\par\noindent {\em Proof:}}{\hfill$\Box$\medskip}
\theoremstyle{plain}

\begin{document}
\begin{abstract}
Let $k$ be a field of characteristic zero. Let $m$ and $\alpha$ be positive integers. For $n\geq 2$, let  $R_n=k[x_1,x_2,\dots,x_n]$ with the $k$-derivation  $d_n$ given by $d_n=(1-x_1 x_2^{\alpha})\partial_{x_1}+x_1^m\partial_{x_2}+x_2\partial_{x_3}+\dots+x_{n-1}\partial_{x_n}$. 
We prove that for integers $m\geq 2$ and $\alpha \geq 1$, 
$d_n$ is a simple $k$-derivation of $R_n$ and $d_n(R_n)$ contains no units. This generalizes a result of D. A. Jordan \cite{j81}. 
We also show that the isotropy group of $d_n$ is conjugate to a subgroup of translations.
\end{abstract}

\maketitle
\section{Introduction}
Let $k$ be a field of characteristic zero and let $R$ be a commutative $k$-algebra. A $k$-linear map $d:R\rightarrow R$ is called to be a \emph{$k$-derivation} if $d(ab)=ad(b)+d(a)b$ for all $a,b\in R$. A $k$-derivation $d$ of $R$ is said to be \emph{simple} if $R$ does not have any proper non-zero ideal $I$ such that $d(I)\subseteq I$. Simple $k$-derivations have several applications such as in proving simplicity of Ore extensions \cite{g89}, in constructing examples of simple Lie rings and non-commutative simple rings \cite{g89}, and they also provide examples of non-holonomic irreducible modules over Weyl algebras \cite{c07}. 

Let $n\geq2$ and let $R_n:=k^{[n]}$ be the polynomial ring in $n$ variables over $k$.
For $n=2$, many examples of simple $k$-derivations of $R_2$ are known; see, for example, \cite{g09, k12,k13,n94,n08,ap23,y19}. Fewer examples of simple $k$-derivations are known for the case of higher number of variables, see  \cite{m01,n94}. Most of these simple $k$-derivations $d$ are such that the $d(R_n)$ contains $1$. It is rare to find examples of simple $k$-derivations such that their images do not contain any unit of $R_n$. Some examples of such derivations are discussed in \cite{j81,k12}. One of the motivations to study such examples is the study of skew polynomial rings $R[x,d]$, which is a non-commutative structure arising from $R$ with a simple $k$-derivation $d$ such that $d(R)$ does not contain any unit of $R$, see  \cite{g89} . 
In this article, we give a class of simple $k$-derivations of $R_n$ such that their images do not contain any unit of $R_n$. 

Let $m\geq 2$ and $\alpha\geq 1$ be integers. % such that $(\alpha+1)$ divides $(m\alpha+1)$. 
 Let $R_n=k[x_1,x_2,\dots,x_n]$ with the $k$-derivation  $d_n$ given by
	\begin{align*}
		d_n:=(1-x_1 x_2^{\alpha})\partial_{x_1}+x_1^m\partial_{x_2}+x_2\partial_{x_3}+\dots+x_{n-1}\partial_{x_n}.
	\end{align*}
    
In 1981, D. A. Jordan proved that for $m=3$ and $\alpha=1$, the $k$-derivation $d_n$ on $R_n$ is simple and $d_n(R_n)$ does not contain any unit of $R_n$.   
Generalizing Jordan's result, we prove the following:
\begin{thm}\label{intro-main-thm}{(\Cref{lstthm})}
 Let $m\geq 2$ and $\alpha\geq 1$ be integers. Then the $k$-derivation $d_n$ on $R_n$ is simple and $d_n(R_n)$ does not contain any unit of $R_n$.
\end{thm}

In \Cref{iso-section}, we study the isotropy group of the $k$-derivation $d_n$ on $R_n$. 
For the ring $R_n$ with a $k$-derivation $d$, the isotropy group is defined as:
$$\text{Aut}(R_n)_{d}:=\{\rho \in \text{Aut}(R_n)| d \rho =\rho d \},$$ 
where $\text{Aut}(R_n)$ is the $k$-automorphism group of $R_n$. 
For $R_n$ and $d_n$ as in \Cref{intro-main-thm},
we show that $\text{Aut}(R_n)_{d_n}$ is conjugate to a subgroup of translations (see \Cref{iso_group}). The proof of the theorem is motivated by a similar result due to D. Yan where $m=3$ and $\alpha=1$ (see \cite[Theorem 2.4]{Y22}). Our result gives a positive result in favor of \cite[Conjecture 1.2]{Y22} which states that for any simple $k$-derivation $d$ on $R_n$, the isotropy group $\text{Aut}(R_n)_{d}$ is conjugate to a subgroup of translations.

\section{Two variables case}

Let $k$ be a field of characteristic zero. 
Let $R_2=k[x,y]$, where $x$ and $y$ are indeterminates over $k$, and $d=y^m \partial_x + (1-x^\alpha y) \partial_y$ be a $k$-derivation of $R_2$, where $m\geq 2$ and $\alpha \geq 1$ are integers. In this section, we show that the image of $d$ does not contain any non-zero element of the form $ax+b$ where $a,b\in k$ (cf., \Cref{general_prop}). For any $h(x)\in k[x]$, we denote the derivative of $h(x)$ by $h'(x)$.

Let us consider $l-m(j-1)$ where $j$ varies over all integers. %, i.e., $j\in \mathbb{Z}$. 
Since $l>0$, for $j=1$ we get that $l-m(j-1)=l>0$. Thus there exists a positive integer such that $l-m(j-1)>0$. Also, as $l$ and $m$ are positive integers, for $j>>0$ it follows that $l-m(j-1)<0$. Thus, there exists greatest positive integer $j_0$ such that $l-m(j_0-1)>0$.

 First, we prove \Cref{l=mj} and \Cref{ML}, which will be used in \Cref{general_prop}. 
 For these lemmas, we adhere to the following notations.

 \begin{nota}\label{second-sec-not} 
 Let $l>0$ and $r=\sum_{i=0}^{l}f_i(x)y^i\in R_2$ be such that $f_i(x)\in k[x]$ for $0\leq i \leq l$, $f_l(x)\neq 0$, and $d(r)=ax+b$ for some $a, b \in k$. Set $f_i(x)=0$ for all integers $i>l$ and $i< 0$. 
\end{nota}

Note that the case $l=0$ will be dealt separately in the proof of \Cref{general_prop}.

\begin{lem}\label{l=mj}
	Let $a, b, r, l$ and $f_i(x)$ be as in \Cref{second-sec-not}. Let $j_0$ be the greatest integer such that $l-m (j_0-1)>0$. Then $l=m j_0$ and for all $j$, $1\leq j\leq j_0$, and all $s$, $1\leq s \leq m-1$,  $\deg_x f_{l-mj}(x)=j(\alpha +1)$ and $\deg_x f_{l-mj+s}(x)\leq (j-1) (\alpha 
    +1)$.
\end{lem}

    \begin{proof}
Since $r=\sum_{i=0}^{l}f_i(x)y^i$ where $f_i(x)\in k[x]$, and $d(r)=ax+b$, we have 
	\begin{align*}
		&(y^m\partial_x+(1-x^\alpha y)\partial_y)(\sum_{i=0}^{l}f_i(x)y^i)=ax+b.
	\end{align*}
    Since $f_j(x)=0$ for $j>l$ or $j<0$, we get that
	\begin{equation}\label{fsteqn}
		\sum_{i=0}^{l+m}(f'_{i-m}(x)+(i+1)f_{i+1}(x)-ix^\alpha f_i(x))y^i =ax+b.
	\end{equation}      
	 From \Cref{fsteqn}, for $i\geq1$, we have
	\begin{equation}\label{scndeqn}
		f'_{i-m}(x)=ix^\alpha f_i(x)-(i+1)f_{i+1}(x).
	\end{equation} 
    
\textbf{(I)} First, we find the values of $f_l,f_{l-1},\ldots ,f_{l-m+1}$. 

    Substituting $i=l+m,l+m-1,\dots,l+2,l+1$ in \Cref{scndeqn}, we get that
	\begin{align}\label{cnsttrms}
	&f'_l(x)=f'_{l-1}(x)=\dots=f'_{l-m+1}(x)=0 \nonumber \\
	\Rightarrow& f_l(x),\,f_{l-1}(x),\dots,\,f_{l-m+1}(x)\in k.
	\end{align}
Further, without loss of generality, we may assume that $f_l(x)=1$. For $1 \leq s \leq m-1$, let us denote $f_{l-s}(x)$ by $\lambda_s$.\\

\textbf{(II)} Here we find the value and degree of $f_{l-m}(x)$. Putting $i=l$ in \Cref{scndeqn}, we have
	\begin{align*}
		f'_{l-m}(x)=lx^\alpha f_l(x)-(l+1)f_{l+1}(x),
	\end{align*}  
 \begin{align}\label{four}
 	\Rightarrow f_{l-m}(x)=\frac{lx^{\alpha+1}}{\alpha+1}+c_{l-m}, \end{align} where $c_{l-m} \in k,$ 
    %is an arbitrary constant, 
    thus $\deg_x f_{l-m}(x)=\alpha+1$.\\

\textbf{(III)} Note that $j_0\geq 1$. We show that the lemma holds if $j_0=1$.

Assume that $j_0=1$. Then $l \leq m$. If $l <m$ then $f_{l-m}(x)=0$, which is a contradiction to the fact that $\deg_x f_{l-m}(x)=\alpha+1$. Thus $l=m$. Hence from \Cref{cnsttrms} and \Cref{four},  $\deg_x f_{s}(x)\leq 0$ for all $1 \leq s \leq m-1$ and $\deg_x f_0(x)=\alpha+1$. Hence the lemma holds.\\  

\textbf{(IV)} Henceforth, we assume that $j_0\geq 2$. Now, we show that for all $j$, $1\leq j\leq j_0$, and all $s$, $1\leq s \leq m-1$,  $\deg_x f_{l-mj}(x)=j(\alpha +1)$ and $\deg_x f_{l-mj+s}(x)\leq (j-1) (\alpha +1)$. We prove this by induction on $j$. 

Note that $j_0\geq 2$ and $l>(j_0-1)m$. From (I) and (II) we have \[f_l=1,f_{l-1}=\lambda_1,\ldots ,f_{l-m+1}=\lambda_{m-1},f_{l-m}=\frac{lx^{\alpha+1}}{\alpha+1}+c_{l-m}\] where $\lambda_1,\ldots,\lambda_{m-1},c_{l-m}\in k$ and $\deg_xf_{l-m}(x)=\alpha+1$. Hence the statement holds for $j=1$. 

Now, let $t$ be a positive integer such that $1\leq t\leq j_0-1$. Assume that for all $j$ with $1\leq j\leq t$, $\deg_x f_{l-mj}(x)=j(\alpha +1)$ and $\deg_x f_{l-mj+s}(x)\leq (j-1) (\alpha +1)$ for all $1\leq s \leq m-1$. 

Then we show that $\deg_x f_{l-m(t+1)}(x)=(t+1)(\alpha +1)$ and $\deg_x f_{l-m(t+1)+s}\leq t (\alpha +1)$ for all $1\leq s \leq m-1$.

Putting $i=l-tm$ in \Cref{scndeqn}, we get that
	\begin{align*} 
        f'_{l-(t+1)m}(x)=(l-tm)x^{\alpha}f_{l-tm}(x)-(l-tm+1)f_{l-tm+1}(x).
	\end{align*}
    Since $\deg_x f_{l-mt}(x)=t(\alpha +1)$ and $\deg_x f_{l-mt+1}(x)\leq (t-1) (\alpha +1)$, then it follows that $\deg_x f_{l-(t+1)m}(x)= (t+1)(\alpha +1)$. 
    
    Similarly, putting $i=l-tm+s$, where $1\leq s\leq m-1$, in \Cref{scndeqn} we get that 
	\begin{align}\label{s eqn}
		f'_{l-(t+1)m+s}(x)=(l-tm+s)x^{\alpha}f_{l-tm+s}(x)-(l-tm+s+1)f_{l-tm+s+1}(x).
	\end{align} 
    
    Case-1. Suppose $1\leq s\leq m-2$. Then $\deg_x f_{l-tm+s},\deg_x f_{l-tm+s+1} \leq (t-1)(\alpha+1)$. Thus, after integrating \Cref{s eqn} with respect to $x$, it follows that $\deg_x f_{l-(t+1)m+s}(x)\leq t(\alpha +1)$. 

    Case-2. Now, let $s= m-1$. Thus $f_{l-tm+s+1}=f_{l-(t-1)m}$. Then $\deg_x f_{l-tm+s} \leq (t-1)(\alpha+1)$ and $\deg_x f_{l-tm+s+1}= \deg_x f_{l-(t-1)m}=(t-1)(\alpha +1)$. Thus, after integrating \Cref{s eqn} with respect to $x$, it follows that $\deg_x f_{l-(t+1)m+s}(x)\leq t(\alpha +1)$. \\

\textbf{(V)} Lastly, we show that $l=mj_0$. 

By the choice of $j_0$, it follows that $l-mj_0\leq 0$. 

Suppose that $l-mj_0<0$. Then $f_{l-mj_0}(x)=0$, and thus $\deg_x f_{l-mj_0}(x)< 0$. But, from (IV) we have $\deg_x{f_{l-mj_0}(x)}=j_0(\alpha +1)>0$. Thus we get a contradiction. Hence $l=mj_0$.  

This completes the proof.
\end{proof}

    \begin{lem}\label{ML}
	Let $a, b, r, l$ and $f_i(x)$ be as in \Cref{second-sec-not}. Now, as in the proof of \Cref{l=mj}, we get that $f_l(x)\in k$. Furthermore, we can assume that $f_l(x)=1$. Then \[f_i(x)=\frac{ax^{(i-1)\alpha+1}+bx^{(i-1)\alpha}}{i}\] for all $i$, $1\leq i\leq m$. %Furthermore, if $\alpha$ and $m$ are both odd then $b\neq 0$ and  \[f_i(x)=\frac{bx^{(i-1)\alpha}}{i}\] for all $i$, $1\leq i\leq m$.
	\end{lem}
\begin{proof} By \Cref{l=mj}, we have
	\begin{align*}
	&\deg_x f_{l-mj_0}(x)=\deg_x f_0(x)=j_0(\alpha+1)
    \end{align*}
    and
    \begin{align*}
	& \deg_x f_{l-m(j_0-1)}(x) = \deg_x f_m(x)=(j_0-1)(\alpha+1). 
	\end{align*}
	Now, comparing the degree zero terms of y in \Cref{fsteqn}, we get that 
	\begin{equation*}
		f'_{-m}(x)+f_1(x)=ax+b\\
        \Rightarrow f_1(x)=ax+b.
	\end{equation*}
Further, substituting $i=1$ in \Cref{scndeqn}, we have	
\begin{equation*}
		f'_{1-m}(x)=x^{\alpha}f_1(x)-2f_2(x).
        \end{equation*} 
        Since $m\geq 2$ we have $f_{1-m}=0$, and thus
        \begin{equation*}
         f_2(x)=x^{\alpha}f_1(x)=\frac{ax^{\alpha+1}+bx^\alpha}{2}.
	\end{equation*}
	\begin{comment}$i=2$ in \Cref{scndeqn} gives
	\begin{equation*}
		f'_{2-m}(x)=2x^{\alpha}f_2(x)-3f_3(x)\\
		\Rightarrow f_3(x)=\frac{ax^{2\alpha+1}+bx^{2\alpha}}{3}
	\end{equation*}
    \end{comment}
	Similarly for $i=2,\dots,m-1$, we have %$f_4(x),f_5(x),\dots,f_{m-1}(x)$ %respectively, then for $i=m-1$ in \Cref{scndeqn}\\
	\begin{equation*}
		f'_{i-m}(x)=i x^{\alpha+1}f_{i}(x)-(i+1)f_{i+1}(x),
        \end{equation*}
	which further implies that 
    \begin{equation*}
f_{i+1}(x)=\frac{ax^{i\alpha+1}+bx^{i\alpha}}{i+1}.
	\end{equation*}
This completes the proof. 
 \end{proof}

\begin{prop}\label{general_prop}
Let $m\geq 2$ and $\alpha\geq1$ be integers. %such that $(\alpha+1)$ divides $(m\alpha+1)$. 
Let $R_2=k[x,y]$ with the $k$-derivation $d=y^m\partial_x+(1-x^\alpha y)\partial_y$. Let $r\in R_2$ be such that $d(r)=ax+b$ for some $a$, $b$ $\in k$. 
Then $r\in k$ and $a=b=0$. 
\end{prop} 

\begin{proof} 
Let $r\in R_2 \setminus k$ be such that $d(r)=ax+b$ for some $a,b\in k.$
As $r\in R_2$, we can write $r$ as $r=\sum_{i=0}^{l}f_i(x)y^i$ where $l\geq0$, $f_i(x)\in k[x]$ and $f_l(x) \neq 0$. 

Suppose that $l=0$. Then $d(r)=d(f_0(x))=y^mf'_0(x)=ax+b.$ Thus we have $f'_0(x)=0$ which implies that $r=f_0(x)\in k$, which is a contradiction.

Henceforth, we assume that $l>0$. Now, we are in the same setup as \Cref{l=mj} and \Cref{ML}, and so we will use facts and results from those lemmas. As in the proof of \Cref{l=mj}, we have $$f_l(x),\,f_{l-1}(x),\,f_{l-2}(x),\dots,\,f_{l-m+1}(x)\in k.$$ Further, without loss of generality, we assume that $f_l(x)=1$. For $0 \leq s \leq m-1$, let us denote $f_{l-s}(x)$ by $\lambda_s$. 

By \Cref{ML}, we have 
\begin{equation}\label{E1}
    f_m(x)=\frac{ax^{(m-1)\alpha+1}+bx^{(m-1)\alpha}}{m}.
\end{equation}

Considering $j_0$ as in \Cref{l=mj}, we have $l=mj_0$ and 
\begin{equation}\label{E2}
    \deg_x f_m(x)=\deg_x f_{l-m(j_0-1)}(x)=(j_0-1)(\alpha +1).
\end{equation}

 From \Cref{E1} and \Cref{E2}, it follows that $j_0>1$. 

By \Cref{ML}, we get that 
\begin{equation}\label{eqn01}
    f_{i}(x)=\frac{ax^{(i-1)\alpha+1}+bx^{(i-1)\alpha}}{i}
\end{equation}
 for $1\leq i\leq m$. From \Cref{scndeqn}, we have \\
	\begin{equation}\label{scndeqn1}
		f'_{i-m}(x)=ix^\alpha f_i(x)-(i+1)f_{i+1}(x)
	\end{equation} for all $i\geq 1$. \\

 \textbf{(I)} First, we find the values of $f_{l-jm}(x)$ and $f_{l-jm-1}(x)$ for all $1\leq j\leq j_0-1$. 
 
 For $1\leq j \leq j_0-1$, substituting $i=l-(j-1)m$ and $i=l-(j-1)m-1$ in \Cref{scndeqn1}, we have 

\begin{equation}\label{eqn1}
     f'_{l-jm}(x)= (l-(j-1)m)x^{\alpha}f_{l-(j-1)m}(x)- (l-(j-1)m+1)f_{l-(j-1)m+1}(x), 
 \end{equation}
 and 
 \begin{equation}\label{eqn2}
     f'_{l-jm-1}(x)= (l-(j-1)m-1)x^{\alpha}f_{l-(j-1)m-1}(x)- (l-(j-1)m)f_{l-(j-1)m}(x),
 \end{equation}
respectively. 

\vspace{2mm}
\textbf{Claim.} We claim that for all $j$, $1\leq j\leq j_0-1,$
\begin{equation}\label{eqn3}
    f_{l-jm}(x)=A_{l-jm}\;x^{j(\alpha+1)}+ O_{l-jm}(x),
\end{equation}
and 
\begin{equation}\label{eqn4}
    f_{l-jm-1}(x)=A_{l-jm-1}\;x^{j(\alpha+1)} 
	+B_{l-jm-1}\;x^{(j-1)(\alpha+1)+1}+O_{l-jm-1}(x),
\end{equation}
where $A_{l-jm}\in \mathbb{Q}_{>0}$, $B_{l-jm-1} \in \mathbb{Q}_{<0}$, $A_{l-jm-1}\in k$, $O_{l-jm}(x)\in k[x]$ with $\deg_x O_{l-jm}(x) < j(\alpha +1)$ and $O_{l-jm-1}(x)\in k[x]$ with $\deg_x O_{l-jm-1}(x) < (j-1)(\alpha+1)+1$.

We prove the above claim by induction on $j$. Note that $f_{l-1}(x)=\lambda_1$, $f_l(x)=1$ and $f_{l+1}(x)=0$. Substituting $j=1$ in \Cref{eqn1} and then integrating with respect to $x$, we get that  
	\begin{equation}\label{eqn5}
		f_{l-m}(x)=\frac{lx^{\alpha+1}}{\alpha+1}+c_{l-m}, 
	\end{equation}

where $c_{l-m}\in k$. Similarly, substituting $j=1$ in \Cref{eqn2} and then integrating with respect to $x$, we have
    \begin{equation}\label{eqn6}
		f_{l-m-1}(x)=\frac{(l-1)x^{\alpha+1}}{\alpha+1}\lambda_1-lx+c_{l-m-1}
	\end{equation} 
    where $c_{l-m-1}\in k$. Note that $A_{l-m}=\frac{l}{\alpha+1} \in \mathbb{Q}_{>0}$,  $B_{l-m-1}(x)=-l\in \mathbb{Q}_{<0}$, $A_{l-m-1}=\frac{(l-1)\lambda_1}{\alpha +1}\in k$, $O_{l-m}(x)=c_{l-m}\in k[x]$ with $\deg_x O_{l-m}(x) < (\alpha+1)$, and $O_{l-m-1}(x)=c_{l-m-1}\in k[x]$ with $\deg_x O_{l-m-1}(x) < 1$. Hence the above claim holds for $j=1$. 

    Next, as induction hypothesis we assume that the above claim holds for $(j-1)$ where $2\leq j\leq j_0-1$. Next, we show that the claim holds for $j$. 
    
    Substituting the values of $f_{l-(j-1)m}(x)$ from \Cref{eqn3} in \Cref{eqn1} we get that
\begin{multline}\label{eqn7}
 f'_{l-jm}(x) = (l-(j-1)m)x^{\alpha}[A_{l-(j-1)m}x^{(j-1)(\alpha+1)} \\
    + O_{l-(j-1)m}(x)]- (l-(j-1)m+1)f_{l-(j-1)m+1}(x).
\end{multline}
    Then integrating with respect to $x$, we can write
\begin{equation}\label{eqn8}
    f_{l-jm}(x)= \frac{(l-(j-1)m) A_{l-(j-1)m}}{j(\alpha +1)} x^{j(\alpha+1)} + O_{l-jm}(x)
\end{equation}
    for some $O_{l-jm}(x)\in k[x]$. By the induction hypothesis, we have $A_{l-(j-1)m} \in \mathbb{Q}_{>0}$ and $\deg_x O_{l-(j-1)m}(x) < (j-1)(\alpha+1)$, and from \Cref{l=mj} we get that\\ $\deg_x f_{l-(j-1)m+1}(x) \leq (j-2)(\alpha +1)$. Hence, it follows that 
    \begin{equation}
    f_{l-jm}(x)=A_{l-jm}\;x^{j(\alpha+1)}+ O_{l-jm}(x),
\end{equation}
where
    \begin{equation}
        A_{l-jm}:=\frac{(l-(j-1)m) A_{l-(j-1)m}}{j(\alpha +1)} \in \mathbb{Q}_{>0}
    \end{equation}
     and $\deg_x  O_{n-jm}(x) < j(\alpha+1)$. 

Similarly, substituting the values of $f_{l-(j-1)m-1}(x)$ and $f_{l-(j-1)m}(x)$, from \Cref{eqn4} and \Cref{eqn3} respectively, in \Cref{eqn2}, we get that 
\begin{multline}\label{eqn9}
    f'_{l-jm-1}(x)= (l-(j-1)m-1)x^{\alpha}[A_{l-(j-1)m-1}x^{(j-1)(\alpha+1)} 
	\\+B_{l-(j-1)m-1}x^{(j-2)(\alpha+1)+1}+O_{l-(j-1)m-1}(x)]
    - (l-(j-1)m)[A_{l-(j-1)m}x^{(j-1)(\alpha+1)} 
    \\+ O_{l-(j-1)m}(x)].
\end{multline}

Now, integrating \Cref{eqn9} with respect to $x$, we can write

\begin{multline}
    f_{l-jm-1}(x)= \frac{(l-(j-1)m-1) A_{l-(j-1)m-1}}{j(\alpha +1)} x^{j(\alpha +1)} \\+  \frac{(l-(j-1)m-1) B_{l-(j-1)m-1}-(l-(j-1)m)A_{l-(j-1)m}}{(j-1)(\alpha+1)+1} x^{(j-1)(\alpha+1)+1} \\+ O_{l-jm-1}(x) 
\end{multline}
 for some $O_{l-jm-1}(x)\in k[x]$. By the induction hypothesis, we have $A_{l-(j-1)m}\in \mathbb{Q}_{>0}$, $A_{l-(j-1)m-1}\in k$,  $B_{l-(j-1)m-1} \in \mathbb{Q}_{<0}$,  $\deg_x O_{l-(j-1)m}(x) < (j-1)(\alpha +1)$ and $\deg_x O_{l-(j-1)m-1}(x) < (j-2)(\alpha+1)+1$. Hence, it follows that 
 \begin{equation}
    f_{l-jm-1}(x)=A_{l-jm-1}\;x^{j(\alpha+1)} 
	+B_{l-jm-1}\;x^{(j-1)(\alpha+1)+1}+O_{l-jm-1}(x),
\end{equation}
where
 \begin{equation}
     A_{l-mj-1}:= \frac{(l-(j-1)m-1) A_{l-(j-1)m-1}}{j(\alpha +1)} \in k, 
 \end{equation}

\begin{equation}
    B_{l-jm-1}:= \frac{(l-(j-1)m-1) B_{l-(j-1)m-1}-(l-(j-1)m)A_{l-(j-1)m}}{(j-1)(\alpha+1)+1} \in \mathbb{Q}_{<0}
\end{equation} 
and $\deg_x \, O_{l-jm-1}(x) < (j-1)(\alpha +1)+1$. Thus the claim follows. \\

\textbf{(II)} Finally, we show that using expressions of $f_m$ and $f_{m-1}$, from (I), we get a contradiction. 

Note that $j_0>1$. Since $l=j_0m$, putting $j=j_0-1$ in \Cref{eqn3}, we have
\begin{equation}\label{E3}
    f_m(x)=A_{l-(j_0-1)m}x^{(j_0-1)(\alpha+1)}+O_{l-(j_0-1)m}(x)
\end{equation}
 with $A_{l-(j_0-1)m}\in \mathbb{Q}_{>0}$ and $\deg_x O_{l-(j_0-1)m}(x) < (j_0-1)(\alpha+1)$. 
%Also, from Equations \ref{eqn0} and \ref{eqn01}, we have \begin{equation*} f_m(x)=\frac{bx^{(j_0-1)(\alpha+1)}}{m}.\end{equation*}

%Also, comparing the coefficient of $x^{(j_0-1)(\alpha+1)}$, from the expressions of $f_m(x)$, it follows that $b>0$.

Similarly, putting $j=j_0-1$ in \Cref{eqn4}, we get that
\begin{multline}\label{E4}
    f_{m-1}(x)=f_{l-(j_0-1)m-1}(x)=A_{l-(j_0-1)m-1}x^{(j_0-1)(\alpha+1)} 
	\\+B_{l-(j_0-1)m-1}x^{(j_0-2)(\alpha+1)+1}+O_{l-(j_0-1)m-1}(x),
\end{multline}

where $A_{l-(j_0-1)m-1}\in k$, $B_{l-(j_0-1)m-1} \in \mathbb{Q}_{<0}$ and $\deg_x O_{l-(j_0-1)m-1}(x) < (j_0-2)(\alpha+1)+1$.

Recall from \Cref{eqn01}, we have 
\begin{equation}\label{eqn001}
    f_{m}(x)=\frac{ax^{(m-1)\alpha+1}+bx^{(m-1)\alpha}}{m}
\end{equation}
and
\begin{equation}\label{eqn002}
    f_{m-1}(x)=\frac{ax^{(m-2)\alpha+1}+bx^{(m-2)\alpha}}{m-1}
\end{equation}
\vspace{2mm}
 \textbf{Case-1.}  Assume that $a\neq 0$. Then from \Cref{eqn001} and \Cref{E3}, we have  
 \begin{equation}\label{E5}
     (j_0-1)(\alpha +1)=(m-1)\alpha +1\,\, \text{and}\,\, a=m A_{l-(j_0-1)m}\in \mathbb{Q}_{>0}. 
 \end{equation}
 Suppose, if $A_{l-(j_0-1)m-1}\neq 0$, then from \Cref{eqn002} and \Cref{E4}, we get that 
 \begin{equation}\label{E6}
     (j_0-1)(\alpha +1)=(m-2)\alpha +1. 
 \end{equation}
Thus from \Cref{E5} and \Cref{E6}, we get two distinct values of $m$, and therefore we have a contradiction. Hence, $A_{l-(j_0-1)m-1}= 0$. Now, comparing the highest degree coefficient of $x$, from \Cref{eqn002} and \Cref{E4}, it follows that $a= m B_{l-(j_0-1)m-1}\in \mathbb{Q}_{<0}$. But from \Cref{E5} we have $a\in \mathbb{Q}_{>0}$, hence we have a contradiction.  

\vspace{2mm}
\textbf{Case-2.} Assume that $a=0$ and $b\neq 0$.
Then from \Cref{eqn001} and \Cref{E3}, we have 
\begin{equation}\label{EQ1}
    (j_0-1)(\alpha +1)=(m-1)\alpha \,\, \text{and}\,\, b=m A_{l-(j_0-1)m} \in \mathbb{Q}_{>0}.
\end{equation}

Suppose $A_{l-(j_0-1)m-1}= 0$. Then from Equations (\ref{eqn002}) and (\ref{E4}) we have $(j_0-2)(\alpha+1)+1=(m-2)\alpha$. Hence, from  \Cref{eqn002} and \Cref{E4}, it follows that $b=(m-1)B_{l-(j_0-1)m-1}\in \mathbb{Q}_{<0}$. Thus we have a contradiction by \Cref{EQ1}. So $A_{l-(j_0-1)m-1}\neq 0$. Then again from Equations (\ref{eqn002}) and (\ref{E4}) we have $(j_0-1)(\alpha+1)=(m-2)\alpha$. But from \Cref{EQ1}, $(j_0-1)(\alpha +1)=(m-1)\alpha$, which is a contradiction.

\vspace{2mm}
\textbf{Case-3.} Assume that $a=b=0$. Then from \Cref{eqn001} and \Cref{E3}, we get that $f_m(x)=0$ and $\deg_x f_m(x)=(j_0-1)(\alpha +1)$, respectively. Since $j_0>1$, it follows that $(j_0-1)(\alpha +1)>0$, which is a contradiction.

Thus we get a contradiction to the fact that $l>0$. Hence $r\in k$, which further implies that $a=b=0$. 
\end{proof}

\begin{rem}
For $m=1$ and any integer $\alpha \geq 1$, considering $r= \frac{x^{\alpha +1}}{\alpha +1} +y\in R_2$, we have $d(r)=1$. Hence \Cref{general_prop} does not hold for $m=1$. 
\end{rem}

\section{The general case} 
Let $m\geq 2$ and $\alpha\geq 1$ be integers. % such that $(\alpha+1)$ divides $(m\alpha+1)$.
Let $R_n=k[x_1,x_2,\dots,x_n]$ with the $k$-derivation  $d_n$ given by
	\begin{align}
		d_n=(1-x_1 x_2^{\alpha})\partial_{x_1}+x_1^m\partial_{x_2}+x_2\partial_{x_3}+\dots+x_{n-1}\partial_{x_n}
	\end{align}
    
    For $n=2$, considering $x_1=y$ and $x_2=x$, we get the ring $R_2=k[x,y]$ with the $k$-derivation $(1-x^\alpha y )\partial_{y}+y^m\partial_{x}$, which has been considered in the previous section. 

\begin{lem}\label{COR}
		Let $m\geq 2$ and $\alpha \geq 1$ be integers. Let $d_2$ denote the $k$-derivation $y^m\partial_x+(1-x^\alpha y)\partial_y$ of $k[x,y]$. Then $d_2$ is a simple $k$-derivation of $k[x,y]$.
\end{lem}
\begin{proof}
	This follows from \cite[Theorem 2.6]{ap23}. 
\end{proof}

\begin{rem}
    Note that the simplicity of $d_2$ in \Cref{COR} can also be derived from \cite[Theorem 4.1 and Theorem 6.1]{k14}.
\end{rem}

The proof of the following \Cref{ML1} and \Cref{lstthm} is along the same lines as \cite[Lemma 2]{j81} and \cite[Theorem 3]{j81} respectively. Nevertheless, we write the proofs for the sake of completion. 

\begin{lem}\label{ML1}
		Let $a, b\in k$, and  $r \in R_n$ be such that $d_n(r)=ax_n+b$. Then $r\in k$ and $a=b=0$.
	\end{lem} 
	\begin{proof}
	     For $n=2$, the lemma holds by \Cref{general_prop}. Suppose that $n>2$, and that the result holds for $n-1$. If $r=0$, the lemma holds. Now let $r\in R_n\setminus \{0\}$. Then we write $r$ as $ r=\sum_{i=0}^{t}f_ix_n^i$ where $f_i\in R_{n-1}$ for all $i$, $0\leq i\leq t$ and $f_{t}\neq 0$. 
         
         Suppose that $t\geq 1$. Then
	\begin{equation}\label{onlyeqn}
d_n(r)=d_n(\sum_{i=0}^{t}f_ix_n^i) 
		=x_n^t	d_n(f_t)+x_n^{t-1}(tx_{n-1}f_t+d_n(f_{t-1}))+ T,
	\end{equation} where $T \in R_{n-1}[x_n]$ with $\deg_{x_n}(T) <t-1$.

Since $d_n(r)=ax_n+b$, \Cref{onlyeqn} implies that $d_{n-1}(f_t)=d_n(f_t) \in k$ and $tx_{n-1}f_t+d_n(f_{t-1}) \in k$. Since $f_t \in R_{n-1}$, by induction hypothesis, $f_t \in k$. Then $d_{n-1}(f_{t-1})=d_{n}(f_{t-1})$ is of the form $a_1x_{n-1}+b_1$ where $a_1,b_1 \in k$. Thus by induction hypothesis, $f_{t-1}\in k$. Thus $d_n(f_{t-1})=0$, which implies that $tx_{n-1}f_t \in k$, which is a contradiction since $f_t\neq0$ and $t\geq 1$.
         
Therefore $t=0$, that is, $r\in R_{n-1}$. Then $d_{n}(r)=d_{n-1}(r)\in R_{n-1}$. Since $d_{n-1}(r)=ax_n+b$, we conclude that $a=0$. Rewrite $d_{n-1}(r)=0\cdot x_{n-1}+b$. Then by the induction hypothesis it follows that $r\in k$ and $b=0$.
\end{proof}

\begin{thm}\label{lstthm}
		For $n\geq 2$, the $k$-derivation $d_n$ is a simple $k$-derivation of $R_n$ and $d_n(R_n)$ does not contain any unit of $R_n$.
 \end{thm} 
\begin{proof} 
    By \Cref{ML1}, it follows that $d_n(R_n)$ does not contain any unit of $R_n$.
    
    Next by induction on $n$, we prove that $d_n$ is a simple $k$-derivation of $R_n$. 
    By \Cref{COR}, $d_2$ is a simple $k$-derivation of $R_2$. Suppose that $n>2$ and that $d_{n-1}$ is a simple $k$-derivation of $R_{n-1}$. 
    View $R_n=R_{n-1}[x_n]$. Let $I$ be a non-zero ideal of $R_n$ such that $d_n(I)\subseteq I$. Let $t= \min\{\deg_{x_n}r| r\in I\; \text{and}\; r\neq 0\}.$ 
    Let $J\subseteq R_{n-1}$ be the set consisting of $0$ and the leading coefficients of non-zero elements of $I$ of degree $t$ in $x_n$. 
    Then $J$ is a non-zero ideal of $R_{n-1}$. For any  non-zero $f_t\in J$, there exists $r\in I$ such that $\deg_{x_n} r =t$ and $r=\sum_{i=0}^{t}f_ix_n^i$ where $f_i\in R_{n-1}$ for $i$, $0\leq i\leq t$. 
    Then $d_n(r)=x_n^td_n(f_t)+x_n^{t-1}(tx_{n-1}f_t+d_n(f_{t-1}))+T\in I$ for some $T \in R_{n-1}[x_n]$ with $\deg_{x_n}(T) <t-1$. 
    Thus $d_{n-1}(f_t)=d_n(f_t)\in J$, and therefore $d_{n-1}(J)\subseteq J$. Since $I\neq 0$, we have $J\neq 0$. 
    Then by induction hypothesis $J=R_{n-1}$. Hence $I$ has an element $\tilde{r}$ whose degree in $x_n$ is minimal among degrees of non-zero elements of $I$, and whose leading coefficient is $1$. 
    Let $\tilde{r}=\sum_{i=0}^{t}g_ix_n^i$ where $g_i\in R_{n-1}$ for $i$, $0\leq i\leq t$ and $g_t=1$. Since $d_n(g_t)=0$, it follows that $\deg_{x_n} d_n(\tilde{r})< \deg_{x_n}\tilde{r}$. 
    Since  $d_n(I)\subseteq I$ and $\tilde{r}\in I$ is a non-zero element of degree $t$, it follows that $d_n(\tilde{r})=0$. Then by \Cref{ML1}, we have $\tilde{r}\in k$. Hence $I=R_n$ and we conclude that $d_n$ is simple. 
    \end{proof}

\section{Isotropy group of $d_n$}\label{iso-section}
%For the rest of the article, we assume that $k$ is algebraically closed of characteristic zero.
Let $d$ be a $k$-derivation of $R_n$. Recall that the isotropy group of the $k$-derivation $d$ is given by
$$\text{Aut}{(R_n)}_{d}=\{\rho \in \text{Aut}(R_n)| d \rho =\rho d \},$$ 
where $\text{Aut}(R_n)$ is the $k$-automorphism group of $R_n$. 
Note that any $k$-automorphism $\rho$ of $R_n$ can be represented by $(g_1,\dots, g_n)$ where $g_i \in R_n$ and $\rho(x_i)=g_i$ for all $i$, $1\leq i\leq n$. 

In \cite[Theorem 1]{MP17}, assuming that $k$ is algebraically closed, it has been proved that the isotropy group $\text{Aut}(R_2)_d=\{id\}$ for any simple $k$-derivation $d$ on $R_2$. However, this also implies that the same result holds over any characteristic zero field. For $n\geq 3$, D. Yan proved that 
$\text{Aut}(R_n)_{d}=\{(x_1,x_2,\dots,x_{n-1},x_n+c)|c\in k\}$ for the simple $k$-derivation $d=(1-x_1x_2)\partial_{x_1}+x^3_1\partial_{x_2}+x_2\partial_{x_3}+\cdots +x_{n-1}\partial_{x_n}$ on $R_n$ (\cite[Theorem 2.4]{Y22}). Furthermore, she conjectured the following:
\begin{conj}\label{conj}
    If $d$ is a simple $k$-derivation of $R_n$ then $\text{Aut}(R_n)_{d}$ is conjugate to a subgroup of translations.
\end{conj}
Some examples of simple $k$-derivations providing positive answers to \Cref{conj}  can be found in \cite{HY23}. In this section, we prove that the above conjecture holds for the simple $k$-derivation $d_n$ of $R_n$. The proof follows along the same line as \cite[Theorem 2.4]{Y22}.

We recall a basic fact here. 
\begin{lem}\label{simple_prop}
   Let $d$ be a simple $k$-derivation of $R_n$. If $d(r)=0$ for some $r\in R_n$ then $r\in k$. 
\end{lem}
%\begin{proof}Let $r\tilde{r}$ be any element in $rR_n$, the ideal generated by $r$. Then $d(r\tilde{r})=d(r)\tilde{r}+rd(\tilde{r})=rd(\tilde{r})$. Thus $d(r\tilde{r})\in rR_n$. Since $d$ is simple, $rR_n$ must be either the zero ideal or the whole ring $R_n$, which implies that either $r=0$ or $r\in k^{\ast}$. Thus $r\in k$.\end{proof}

Our aim is to prove the following:
\begin{thm}\label{iso_group}
  Let $m\geq2$ and $\alpha \geq 1$ be integers. %  such that $(\alpha+1)$ divides $(m\alpha+1)$. 
  For $n\geq 3$, let $d_n=(1-x_1 x_2^{\alpha})\partial_{x_1}+x_1^m\partial_{x_2}+x_2\partial_{x_3}+\dots+x_{n-1}\partial_{x_n}$  be the $k$-derivation of $R_n=k[x_1, \dots,x_n ]$. Then $\text{Aut}(R_n)_{d_n}=\{(x_1,x_2,\dots , x_{n-1}, x_n+c)| c\in k\}$ .   
\end{thm}

We first prove the following:
\begin{lem}\label{x_1x_2_fixed}
Let $n\geq3$. Let $R_n$ and $d_n$ be as above. Let $\rho \in \text{Aut}(R_n)_{d_n}$. Then $\rho(x_1)=x_1$ and $\rho(x_2)=x_2$. 
\end{lem}
\begin{proof}
Let $\rho \in \text{Aut}(R_n)_{d_n}$. Then $\rho d_n=d_n \rho$. 
Let $i\in \{3,\dots, n\}$. Suppose that $\displaystyle \rho(x_1)= \sum_{l=0}^{t} f_{l} x^l_{i}$ and $ \displaystyle \rho(x_2)=\sum_{j=0}^{s} g_{j} x^{j}_{i},$
where $f_t,g_s \neq 0$, $f_l,g_j \in k[x_1,\dots,\hat{x_{i}},\dots, x_n]$ for $0\leq l\leq t$ and $0\leq j \leq s$. 
Note that \begin{equation}\label{4.4.2}
    d_n(f_t)= (1-x_1x^{\alpha}_2) \partial_{x_1} f_t + x^{m}_1 \partial_{x_2} f_t+\cdots +x_{i-2}\partial_{x_{i-1}} f_t+ x_i \partial_{x_{i+1}} f_t+\cdots +x_{n-1}\partial_{x_n} f_t.
\end{equation}
Since $d_n \rho(x_1)=\rho d_n(x_1)$, we have:
\begin{equation}\label{4.4.1}
    d_n(f_t)x^t_i+\cdots + d_n(f_0)+x_{i-1}(tf_tx^{t-1}_i+\dots +f_1) = 1-(f_tx^t_i+ \cdots +f_0) (g_sx^s_i+\cdots +g_0)^{\alpha}.
\end{equation}
Since $\deg_{x_i}(d_n(f_t))\leq 1$, we get that $s\leq 1$.

Suppose $s= 1$. Then it follows from  \Cref{4.4.1} and \Cref{4.4.2} that $\alpha =1$. From  \Cref{4.4.1} we get that  $\partial_{x_{i+1}}f_t=-f_tg_s$, which is a contradiction since $f_tg_s\neq0$ and $\deg_{x_{i+1}} (\partial_{x_{i+1}}f_t) < \deg_{x_{i+1}}f_t$. 
Thus $s=0$. Hence $\rho(x_2)\in k[x_1,\dots,\hat{x_{i}},\dots, x_n]$ for all $i$, $3\leq i \leq n$. Therefore $\rho(x_2)\in k[x_1,x_2]$. 

Let $i\in \{3,\dots, n\}$. Since $\rho d_n(x_2)=d_n \rho(x_2)$, we have 
\begin{equation}\label{4.4.3}
    (f_tx^t_i+\cdots +f_0)^m=(1-x_1 x^{\alpha}_2)\partial_{x_1}\rho(x_2)+x^m_1\partial_{x_2}\rho(x_2).
\end{equation}
Comparing $\deg_{x_i}$ in \Cref{4.4.3}, we get that $t=0$. Thus for all $i$, $3\leq i \leq n$,  $\rho(x_1)\in k[x_1,\dots,\hat{x_{i}},\dots, x_n]$. Hence $\rho(x_1)\in k[x_1,x_2]$. 

Note that $\rho^{-1}\in Aut(R_n)_{d_n}$. Similarly $\rho^{-1}(x_1), \rho^{-1}(x_2)\in k[x_1,x_2].$ Thus $\rho\in Aut(k[x_1,x_2])_{d_2}$. Since $d_2$ is simple by \Cref{lstthm}, % and $k$ is algebraically closed of characteristic zero,
it follows from \cite[Theorem 1]{MP17} that $\rho(x_1)=x_1$ and $\rho(x_2)=x_2$. \end{proof}

\textbf{Proof of \Cref{iso_group}} For any $c\in k,$ it is easy to see that $(x_1,.., x_{n-1}, x_n+c)\in \text{Aut}(R_n)_{d_n}$.
%For $n=2$, the result follows from \cite[Theorem 1]{MP17}.Assume that $n\geq 3$.

Let $\rho\in Aut(R_n)_{d_n}$. 
By \Cref{x_1x_2_fixed}, $\rho(x_1)=x_1$ and $\rho(x_2)=x_2$. 
Let $\rho(x_3)=f_{t}x^t_n+\cdots +f_1x_n+f_0$ where $f_t\neq 0$ and $f_j\in k[x_1,\dots, x_{n-1}]$ for $0\leq j\leq t$. 
Since $\rho d_n(x_3)=d_n \rho(x_3),$ we have 
\begin{equation}
    x_2=d_n(f_t)x^t_n+\cdots +d_n(f_1)x_n+d_n(f_0)+x_{n-1}(tf_tx^{t-1}_n+\cdots +f_1).
\end{equation}
We first show that $\rho(x_3)\in k[x_1,x_2,x_3]$. This statement clearly holds for $n=3$.
Suppose $n>3$. Assume $t\geq1$. Then equating the coefficients of $x^j_n$ for $j= t,\dots,1,0$, we have 
\begin{equation}\label{4.3.5}
    d_n(f_t)=0
\end{equation}
\begin{equation}\label{4.3.6}
    d_n(f_{t-1})+tf_tx_{n-1}=0
\end{equation}
\begin{equation*}
\begin{aligned}
{\vdots} 
\end{aligned}
\end{equation*}
\begin{equation}\label{4.3.7}
    d_n(f_1)+2f_2x_{n-1}=0
\end{equation}
\begin{equation}\label{4.3.8}
    d_n(f_0)+f_1x_{n-1}=x_2.
\end{equation}
Since $d_n$ is simple by \Cref{lstthm}, it follows from \Cref{simple_prop} that $f_t\in k^{\ast}$. If $t\geq2$ then $d_n(f_{t-1}+t f_t x_n)=0$ by \Cref{4.3.6}. By \Cref{simple_prop} we get that $f_{t-1}+t f_t x_n\in k$, which is a contradiction. Now let $t=1$. By \Cref{4.3.5} and \Cref{simple_prop}, $f_1\in k^{\ast}.$ Then by \Cref{4.3.8} we have $d_n(f_0+f_1 x_n-x_3)=0$. Again using \Cref{simple_prop}, we get that $f_0+f_1 x_n-x_3\in k$, which is a contradiction. Thus $t=0$. Hence $\rho(x_3)\in k[x_1,x_2,\dots, x_{n-1}]$. Repeating the above procedure it follows that $\rho(x_3)\in k[x_1,x_2,x_3]$. 

Let $\rho(x_3)=g_tx^t_3+\cdots+ g_1x_3+g_0$ where $g_t\neq 0$, $g_j\in k[x_1,x_2]$ for $0\leq j\leq t$. Using $d_n \rho(x_3)=\rho d_n(x_3)$ and arguing as above, we get that $t\leq 1$. 
Suppose that $t=0$. Then $d_n(g_0-x_3)=0$. By \Cref{simple_prop} we get that $g_0-x_3\in k$, which is a contradiction. Thus $t=1$ and we have
\begin{equation}
    d_n(g_1)=0
\end{equation}
\begin{equation}
    d_n(g_0)+g_1x_2=x_2.
\end{equation}
Thus by \Cref{simple_prop}, $g_1\in k^{\ast}$ and $g_0+g_1x_3-x_3\in k$. Therefore $g_1=1$ and $g_0\in k$. Hence $\rho(x_3)=x_3+c_3$ where $c_3:=g_0\in k$. This completes the proof for $n=3$ case. 

Now assume that $n=4$. 
Arguing similarly as above and using \Cref{lstthm}, we get that $\rho(x_4)=f_1x_4+f_0$, where $f_1 \neq 0$ and $f_0,f_1 \in k[x_1,x_2,x_3]$. Since $d_n \rho (x_4) = \rho d_n (x_4)$, it follows that $f_1\in k^{\ast}$ and $d_n(f_0+f_1x_4-x_4)=c_3$ where $c_3\in k$. If $c_3\in k^{\ast}$, then we have a contradiction by \Cref{lstthm}. Hence $c_3=0$. Furthermore, using \Cref{simple_prop}, we have $f_1=1$ and $f_0\in k$. Thence $\rho(x_3)=x_3$ and $\rho(x_4)=x_4+c_4$ where $c_4:=f_0\in k$. This completes the proof for $n=4$.

Now assume that $n>4$. Proceeding similarly, we get that for all $1\leq j\leq n-1$, $\rho(x_j)=x_j$ and $\rho(x_n)=x_n+c$ for some $c\in k$.  $\hfill{\square}$

\section*{Declaration of interests} The authors declare that they have no known competing financial interests or personal relationships that could have appeared to influence the work reported in this paper.

\section*{Acknowledgements} The first named author would like to thank the Department of Science and Technology of India  for the INSPIRE Faculty fellowship grant IFA23- MA 197 and IIT Indore for Young Faculty Seed Grant IITIIYFRSG 2024-25/Phase-V/04R. The second author would like to thank IIT Indore for the Young Faculty Research Grant (No. IITIIYFRSG 2024-25/Phase-VII/05). The third named author is financially supported by the SRF grant from CSIR India, Sr. No. 09/1022(16075)/2022-EMR-I. The authors express their sincere gratitude to the reviewers for the insightful and constructive comments, which have significantly contributed to enhancing the clarity and readability of the manuscript.

\end{document}